\documentclass{imatma}
\jno{xxx000}
\received{}
\revised{}

\usepackage{pstricks}
\usepackage{amssymb,amsfonts}
\usepackage{amsmath}
\usepackage{amsthm}
\usepackage{enumerate}
\usepackage{mathtools}
\usepackage{graphicx}
\usepackage{blkarray}
\usepackage{bm}
\usepackage{authblk}
\usepackage{natbib} 
\usepackage{xpatch}

\newcommand{\myAd}{\mathcal{A}\!\!d}

\newcommand{\upL}{\mathrm{L}}

\newcommand{\ds}{\displaystyle}

\newcommand{\bkappa}{\bm{\kappa}}
\newcommand{\bom}{\bm{\sigma}}

\newcommand{\eubu}{\mathrm{E}_\mbu}

\newcommand{\eubka}{\mathrm{E}_{\bkappa}}

\newcommand{\es}{\mathrm{S}}
\newcommand{\iup}{\iota}
\newcommand{\id}{\mathrm{id}}

\newcommand{\mbu}{\mathbf{u}}

\newcommand{\man}{\mathcal{M}}

\newcommand{\ddt}{\ds\frac{{\rm d}}{\rm{d}t}}

\begin{document}
\title{Moving Frames and Noether's  Finite Difference Conservation Laws II.}
\shorttitle{Noether's Finite Difference Conservation Laws II}
\author{%
{\sc
E. L. Mansfield\thanks{Corresponding author. Email: e.l.mansfield@kent.ac.uk},  
A. Rojo-Echebur\'{u}a}\thanks{Email: arer2@kent.ac.uk}\\ \vskip2pt
SMSAS, University of Kent, Canterbury, CT2 7FS, UK
}
\shortauthorlist{E.L. Mansfield \emph{et al.}}

\maketitle

\begin{abstract}
{In this second part of the paper, we consider finite difference Lagrangians which are invariant under linear and projective actions of  $SL(2)$, and the linear equi-affine action which preserves area in the plane.  

We first find the generating invariants, and then use the results of the first part of the paper to write the Euler--Lagrange difference equations and  Noether's difference conservation laws for any invariant Lagrangian, in terms of the invariants and a difference moving frame. We then give the details of the final integration step, assuming the Euler Lagrange equations have been solved for the invariants. This last step relies on understanding the Adjoint action of the Lie group on its Lie algebra. We also use methods to integrate Lie group invariant difference equations developed in Part I. 

Effectively, for all three actions, we show that solutions to the Euler--Lagrange equations, in terms of the original dependent variables, share a common structure for the whole set of  Lagrangians invariant under each given group action, once the invariants are known as functions on the lattice.}
{Noether's Theorem, Finite Difference, Discrete Moving Frames}
\end{abstract}


\tableofcontents

\section{Introduction}

This is a continuation of \emph{Moving Frames and Noether's  Finite Difference Conservation Laws I} \cite{newpaper} where now we consider Lie group actions of the special linear group $SL(2)$ and $SL(2) \ltimes {\mathbb{R}}^2$. The Lie group  $SL(2)$ is the set of $2\times 2$ real (or complex) matrices with determinant equal to unity. Its typical element is written as 
\begin{equation}\label{typSL2} g=\left(\begin{array}{cc} a &b \\ c & d \end{array} \right),\qquad ad-bc=1.\end{equation}
Its Lie algebra, the set of $2\times 2$ real (or complex) matrices with zero trace, is denoted $\mathfrak{sl}(2)$. Part I of this paper 
developed all the necessary theory and considered simpler solvable groups; here we show some additional techniques needed for semi-simple groups. Smooth variational problems with an $SL(2)$ and $SL(2) \ltimes {\mathbb{R}}^2$ symmetry were considered using moving frame techniques in \cite{GonMan}, \cite{GonMan3} and \cite{Mansfield:2010aa}. Here we show the finite difference analogue for these variational problems. 

{We first recall the notation, definitions and the main Theorems developed in the first part of this paper (\cite{newpaper}), which we will need; citations to the literature on standard notions from the discrete calculus of variations and on moving frames can be found in the Introduction to Part 1.}

{We consider the dependent variables to take values in $U\subset \mathbb{R}^q$ with coordinates  $\mathbf{u}=(u^1,\dots, u^q)$. 
We use the notation $u^\alpha_j=u^\alpha(n+j)$ for $\alpha = 1,\cdots,q$ and $n,j \in \mathbb{Z}$. } 
The (forward) shift operator $\es$ acts on functions of $n$ as follows:
\[
\es:n\mapsto n+1,\qquad \es:f(n)\mapsto f(n+1),
\]
for all functions $f$ whose domain includes $n$ and $n+1$. In particular, $$\es:u^\alpha_j\mapsto u^\alpha_{j+1}$$
on any domain where both of these quantities are defined.
The forward difference operator is $\es-\id$, where $\id$ is the identity operator:
\[
\id:n\mapsto n,\qquad \id:f(n)\mapsto f(n),\qquad \id:u^\alpha_j\mapsto u^\alpha_j.
\]
The shift $\es$ is an operator on $P^{(-\infty,\infty)}_n(U)$
where
$P^{(J_0,J)}_n(U)\simeq U\times\cdots\times U$ ($J-J_{0}+1$ copies) with coordinates $z=(\mbu_{J_{0}},\dots,\mbu_J)$, where $J_{0}\leq 0$ and $J\geq 0$. 
We denote the $j^{\mathrm{th}}$ power of $\es$ by $\es_j$, so that $\mbu_j=\es_j\mbu_0$ for each $j\in\mathbb{Z}$.

{We will consider actions that are assumed to be free and regular on a manifold $M$. Therefore, there exists a cross section $\mathcal{K}\subset M$ that is transverse to the orbits $\mathcal{O}(z)$ and, for each $z\in M$, the set $\mathcal{K}\cap \mathcal{O}(z)$ has just one element, the projection of $z$ onto $\mathcal{K}$.
Using the  cross-section $\mathcal{K}$, a moving frame for the group action on {{a neighbourhood $\mathcal{U}\subset M$ of $z$}} can be defined as follows.
\begin{definition}[Moving Frame]
Given a smooth Lie group action $G\times M\rightarrow M$, a moving frame is an equivariant map $\rho: \mathcal{U}\subset M\rightarrow G$. Here $\mathcal{U}$ is called the domain of the frame.
\end{definition}
A left equivariant map satisfies $\rho(g\cdot z) = g\rho(z)$, and a right equivariant map satisfies $\rho(g\cdot z)=\rho(z)g^{-1}$. 
In order to find the frame, let the cross-section $\mathcal{K}$ be given by a system of equations $\psi_r(z)=0$, for $r=1,\ldots,R,$ where $R$ is the dimension of the group $G$. One then solves the
so-called normalization equations,
\begin{equation}\label{frameEqA}
\psi_r(g\cdot z) = 0, \qquad r=1,\ldots, R,
\end{equation}
for $g$ as a function of $z$.
The solution is the group element $g=\rho(z)$ that maps $z$ to its projection on $\mathcal{K}$.
In other words, the frame $\rho$ satisfies
$$
\psi_r(\rho(z)\cdot z)=0, \qquad r=1,\ldots, R. 
$$
}
{\begin{lemma}[Normalized Invariants]
Given a left or right action $G\times M \rightarrow M$ and a right frame $\rho$, then
 $\iup(z)=\rho(z)\cdot z$, for $z$ in the domain of the frame $\rho$, is invariant under the group action.
\end{lemma}
\begin{definition} The normalized invariants are the components of $\iup(z)$.\end{definition}
\begin{theorem}[Replacement Rule]\label{reprule} 
If $F(z)$ is an invariant of the given action $G\times M\rightarrow M$
for a right moving frame $\rho$ on $M$, then $F(z)=F(\iup(z))$.
\end{theorem}
\begin{definition}[Invariantization Operator]\label{invOpDef} Given a right moving frame $\rho$, the map $z\mapsto \iup(z)=\rho(z)\cdot z$ is called the \emph{invariantization operator}. This operator extends to functions as
$f(z)\mapsto f(\iup(z))$, and $f(\iup(z))$ is called the \emph{invariantization} of $f$.\end{definition}
\noindent If $z$ has components $z^\alpha$, let $\iup(z^\alpha)$ denote the $\alpha^{\mathrm{th}}$ component of $\iup(z)$.}

\bigskip
{
A discrete moving frame is a sequence of moving frames $(\rho_k)$, $k=1,\dots, N$ with a nontrivial intersection of domains which, locally, are uniquely determined by the cross-section $\mathcal{K}=(\mathcal{K}_1,\dots,\mathcal{K}_N)$ to the group orbit through $z\in M^N$, where $M^N$ is the Cartesian product manifold.}
For a right discrete frame, we define the invariants
\begin{equation}\label{invariant}
I_{k,j}:= \rho_k(z)\cdot z_j.
\end{equation}
Suppose that $M$ is $q$-dimensional. Therefore $z_j$ has components $z_j^1,\dots,z_j^q$,  and the $q$ components of $I_{k,j}$ are the invariants
\begin{equation}
I^\alpha_{k,j}:= \rho_k(z)\cdot z^\alpha_j,\qquad \alpha=1,\dots q.
\end{equation}
We will denote the invariantization operator with respect to the frame $\rho_k(z)$ by $\iup_k$,
so that $$I_{k,j}=\iup_k(z_j),\qquad I_{k,j}^\alpha=\iup_k(z_j^\alpha).$$
\begin{definition}[Discrete Maurer--Cartan invariants] \label{KmatDef} Given a right discrete moving frame $\rho$, the \emph{right discrete  Maurer--Cartan group elements} are
\begin{equation}\label{Kk}
K_k=\rho_{k+1}\rho_k^{-1}
\end{equation}
where defined.
\end{definition}
We call the components of the Maurer--Cartan elements the \emph{Maurer--Cartan invariants}.

{A difference moving frame is a natural discrete moving frame that is adapted to difference equations by prolongation conditions.
\begin{theorem}\label{diffMCrecThm}
	Given a right difference moving frame $\rho$, the set of all invariants is generated by the set of components of $K_0=\rho_1\rho_0^{-1}$ and $I_{0,0}=\rho_0(z)\cdot z_0$.
\end{theorem}
As $K_0$ is invariant, by \eqref{reprule} we have that
\begin{equation}\label{invk0}
K_0=\iup_0(\rho_1),
\end{equation}
where $\iup_0$ denotes invariantization with respect to the frame $\rho_0$.}

\bigskip

{Given any smooth path $t\mapsto z(t)$ in the space $\man =M^N$, consider the induced group action on the path and its tangent. We extend the group action to the dummy variable $t$ trivially, so that $t$ is invariant.
The action is extended to the first-order jet space of $\man$ as follows:
$$g\cdot\ds\frac{{\rm d} z(t)}{{\rm d}t} = \ds\frac{{\rm d} \left(g\cdot z(t)\right)}{{\rm d}t}\,.$$
If the action is free and regular on $\man $, it will remain so on the jet space and we may use the same frame to find the first-order differential invariants
\begin{equation}\label{FOdiffinvariant}
I_{k,j;\,t}(t):=\rho_k(z(t))\cdot \ds\frac{{\rm d}  z_j(t)}{{\rm d}t}\,.
\end{equation}
\begin{definition}[Curvature Matrix]\label{Nndefn} The curvature matrix $N_k$ is given by
 \begin{equation}\label{curvMatDef}
N_k=\left(\ds\frac{{\rm d}}{{\rm d}t}\, \rho_k\right) \rho_k^{-1}
\end{equation}
when $\rho_k$ is in matrix form.
\end{definition}
It can be seen that for a right frame,  $N_k$ is an invariant matrix that involves the first order differential invariants.
The above derivation applies to all discrete moving frames. For a difference frame the follow syzygy holds
\begin{equation}\label{DiffKeqnN}\ds\frac{{\rm d}}{{\rm d}t} K_0=(\es N_0) K_0-K_0N_0.\end{equation}
As $N_0$ is invariant, from \eqref{reprule} we have that
\begin{equation}\label{invn0}
N_0=\iup_0\!\left(\ds\frac{{\rm d}}{{\rm d}t}\, \rho_0\right).
\end{equation}}

\bigskip

{
In order to calculate the invariantized variation of the Euler--Lagrange equations, we may use the differential--difference syzygy
\begin{equation}\label{mainDDSyz}
\ds\frac{{\rm d}}{{\rm d}t} \bkappa = \mathcal{H}\bom,
\end{equation}
where $\bkappa$ is a vector of generating  invariants, $\mathcal{H}$ is a linear difference operator with coefficients 
that are functions of $\kappa$ and its shifts, and $\bom$ is a vector of generating first order differential invariants
of the form \eqref{FOdiffinvariant}.
Note that \eqref{mainDDSyz} comes from rearranging components in  \eqref{DiffKeqnN}.
\begin{definition}\label{adopdef} Given a linear difference operator $\mathcal{H}=c_j \es_j$,  the adjoint operator $\mathcal{H}^*$ is defined by
$$\mathcal{H}^*(F)=\es_{-j}(c_j F)$$
and the associated boundary term $A_{\mathcal{H}}$ is defined by $$  F\mathcal{H}(G)-\mathcal{H}^*(F)G = (\es-\id)(A_{\mathcal{H}}(F,G)),$$
for all appropriate expressions $F$ and $G$.
\end{definition}
Now suppose we are given a group action $G\times {M}\rightarrow {M}$ and that we have found a difference frame for this action. 
\begin{theorem}[Invariant Euler--Lagrange Equations]\label{mainThmELeqns} ({Theorem 5.2} in \cite{newpaper}).\quad 
Let $\mathcal{L}$ be a Lagrangian functional whose invariant Lagrangian is given in terms of the generating invariants as
$$\mathcal{L}=\sum L(n,\bkappa_0,\dots, \bkappa_{J_1}),$$
and suppose that the differential--difference syzygies are
$$ \ds\frac{{\rm d} \bkappa}{{\rm d}t} = \mathcal{H}\bom.$$
Then (with $\cdot$ denoting the sum over all components)
\begin{equation}\label{ELinvresult}
\eubu(\upL)\cdot \mbu_0'=\big(\mathcal{H}^*\eubka(L)\big)\cdot\bom,
\end{equation}
where $\eubka(L)$ is the difference Euler operator with respect to $\bkappa$.
Consequently, the invariantization of the original Euler--Lagrange equations is
\begin{equation}\label{invEL}
\iup_0\big(\eubu(\upL)\big)=\mathcal{H}^*\eubka(L).
\end{equation}
\end{theorem}
Consequently, the original Euler--Lagrange equations, in invariant form, are equivalent to
\[
\mathcal{H}^*\eubka(L)=0.
\]
\begin{theorem}\label{thmi} (Theorem 7.1 in \cite{newpaper}) Suppose that the conditions of Theorem  \ref{mainThmELeqns} hold. Write
	\[
	 {A}_{\mathcal{H}}=\mathcal{C}^{j}_{\alpha}\es_j(\sigma^{\alpha}),
	\]
where each $\mathcal{C}^{j}_{\alpha}$ depends only on $n, \bkappa$ and its shifts.
Let $\Phi^{\alpha}(\mbu_0)$ be the row of the matrix of characteristics corresponding to the dependent variable $u^{\alpha}_0$ and denote its invariantization by $\Phi^{\alpha}_0(I)=\Phi^{\alpha}(\rho_0\cdot\mbu_0)$. 
Then the $R$ conservation laws in row vector form amount to
\begin{equation}\label{conlawsfinala}
\mathcal{C}_{\alpha}^{j}\es_j\{\Phi^{\alpha}_0(I)\!\:\myAd\left(\rho_{0}\right)\}=0.
\end{equation}
That is, to obtain the conservation laws, it is sufficient to make the replacement
\begin{equation}\label{InvAdrepEqn}
\sigma^\alpha \mapsto \{\Phi^{\alpha}(g\cdot \mbu_0)\!\:\myAd(g)\} \big\vert_{g=\rho_0}.
\end{equation}
in $A_{\mathcal{H}}$.
\end{theorem}}


\section{The linear action of  $SL(2)$ in the plane}

We  consider the  action of  $SL(2)$ on the prolongation space $P_n^{(0,0)}(\mathbb{R}^2)$, which has coordinates $(x_0,y_0)$. This action is given by 
\begin{equation}\label{sl2}
\left(\begin{array}{c} x_0 \\  y_0  \end{array} \right) \mapsto \left(\begin{array}{cc} a &b \\ c & d \end{array} \right) \left(\begin{array}{c} x_0\\  y_0  \end{array} \right) = \left(\begin{array}{c} \widetilde{x_0} \\  \widetilde{y_0}  \end{array} \right), \qquad ad-bc=1.
\end{equation}
\subsection{The Adjoint action}
For our calculations we need the adjoint representation of  $SL(2)$ relative to this group action. The infinitesimal vector fields are
\[
{\bf v}_a=x \partial_{x} - y \partial_{y}, \quad {\bf v}_b=y \partial_{x}, \quad {\bf v}_c=x \partial_{y}.
\]
We have that the induced action on these are
\[
\left(\begin{array}{ccc} \widetilde{{\bf v}_a} &\widetilde{{\bf v}_b} & \widetilde{{\bf v}_c} \end{array}\right) =   \left(\begin{array}{ccc} {\bf v}_a & {\bf v}_b & {\bf v}_c \end{array}\right){\myAd(g)}^{-1}
\]
where
\begin{equation}\label{sl2adj}
\myAd(g)= \begin{blockarray}{cccc} & a & b & c \\ \begin{block}{c(ccc)} a & ad+bc & -ac & bd\\ b & -2ab & a^{2} & -b^{2}\\ c & 2cd&-c^{2}&d^{2}\\\end{block}\end{blockarray} \ .
\end{equation}

\subsection{The discrete frame, the generating invariants and their syzygies}
Taking the normalisation equations $\widetilde{x}_{0}=1$, $\widetilde{x}_{1}=\widetilde{y}_{0}=0$ and solving for $a,b$ and $c$, we define the moving frame
\[
\rho_{0}(x_0,y_0,x_1,y_1)=
\left(
\begin{array}{cc}
\displaystyle\ds\frac{y_1}{\tau}& -\displaystyle\ds\frac{x_1}{\tau}\\[10pt]
-y_{0} & x_{0}
\end{array}
\right)\in SL(2)
\]
where we have set $\tau = x_0y_1-x_1y_0$. Then $\rho_k = \text{S}_k \rho_0$ gives the discrete moving frame $(\rho_k)$.

\subsubsection{The generating discrete invariants}
The Maurer--Cartan matrix is
\begin{equation}\label{MaurerCartan}
K_0={{\iota_0}}(\rho_1)=\left(\begin{array}{cc} \kappa  &\ds\frac{1}{\tau}\\[10pt] -\tau & 0 \end{array}\right)
\end{equation}
where  we have set  $\kappa = \displaystyle\ds\frac{x_0 y_{2} - x_{2} y_0}{x_1 y_{2} - x_{2} y_1}$. Note that $\tau=\iota_0(y_1)$ and $\kappa=\ds\frac{{{\iota_0}}(y_2)}{\text{S}({{\iota_0}}(y_1))}$ are invariant, by the equivariance of the frame.

By the general theory of discrete moving frames, the algebra of invariants is generated by 
$\tau$, $ \kappa$ and their shifts.

\subsubsection{The generating differential invariants}
We now consider $x_j=x_j(t)$, $y_j=y_j(t)$ and we define some first order differential invariants by setting
\begin{equation}\label{fodi}
I^x_{k,j;t}(t)\coloneqq \rho_k \cdot x^{\prime}_j \qquad \text{and} \qquad I^y_{k,j;t}(t) \coloneqq \rho_k \cdot y^{\prime}_j,
\end{equation}
where $x^{\prime}_j=\ddt x_j(t)$ and $y^{\prime}_j=\ddt y_j(t)$. We set the notation
\begin{equation}\label{easenot}
\sigma^x \coloneqq I^x_{0,0;t}(t) \qquad \text{and} \qquad\sigma^y \coloneqq I^y_{0,0;t}(t).
\end{equation}

\begin{lemma}
For all $k$, $j$, both $I^x_{k,j;t}(t)$ and $I^y_{k,j;t}(t)$ may be written in terms of $\sigma^x$, $\sigma^y$, $\kappa$, $\tau$ and their shifts.
\end{lemma}
\begin{proof}
First note that $I^x_{j,j;t}(t)=\text{S}_j\sigma^x$, $I^y_{j,j;t}(t)=\text{S}_j\sigma^y$.
We have next for each $k>j$ that 
$$I^x_{k,j;t}(t)=\rho_k \cdot x^{\prime}_j={\rho_k \rho_{k-1}^{-1}\rho_{k-1}\rho_{k-2}^{-1}\rho_{k-2}\cdots \rho_{j}^{-1} \rho_{j}\cdot x^{\prime}_j}=(\text{S}_kK_0)\cdots (\text{S}_jK_{0}) \text{S}_j\sigma^x$$
while similar calculations hold for for $k<j$ and for $I^y_{k,j;t}(t)$. 
\end{proof}
For our calculations, we need to know $I^x_{0,2;t}(t)$,  $I^x_{0,1;t}(t)$ and $I^y_{0,1;t}(t)$ explicitly. We have
\begin{equation}\label{relation1}
\begin{array}{rcl}
\left(\begin{array}{c}I^x_{0,1;t}(t)\\ I^y_{0,1;t}(t)\end{array}\right)&=&\rho_0\left(\begin{array}{c}x_1^\prime\\ y_1^\prime\end{array}\right)\\
&=&\rho_0\rho_{1}^{-1}\rho_1\left(\begin{array}{c}x_1^\prime\\ y_1^\prime\end{array}\right)\\[12pt]
&=&K_0^{-1} \left(\begin{array}{c}\text{S}\sigma^x\\ \text{S}\sigma^y\end{array}\right)\\[12pt]
&=& \left(\begin{array}{c} -\displaystyle\ds\frac{\text{S}\sigma^y}{\tau}\\[10pt] \tau\text{S}\sigma^x +\kappa\text{S}\sigma^y\end{array}\right)
\end{array}
\end{equation}
while a similar calculation yields, setting $\tau_j=\text{S}_j\tau$ and $\kappa_j=\text{S}_j\kappa$,
\begin{equation}\label{relation2}
\begin{array}{rcl}
\left(\begin{array}{c}I^x_{0,2;t}(t)\\ I^y_{0,2;t}(t)\end{array}\right)&=&K_0^{-1} (\text{S}K_0^{-1}) \left(\begin{array}{c} \text{S}_2\sigma^x\\\text{S}_2\sigma^y\end{array}\right)\\[12pt]
&=&\left(\begin{array}{cc} -\displaystyle\ds\frac{\tau_1}{\tau} &  -\displaystyle\ds\frac{\kappa_1}{\tau}\\[10pt]
\kappa\tau_1 & \kappa\kappa_1 -\displaystyle\ds\frac{\tau}{\tau_1}\end{array}\right)\left(\begin{array}{c} \text{S}_2\sigma^x\\\text{S}_2\sigma^y\end{array}\right).
\end{array}
\end{equation}

We now define
\begin{equation}\label{CMex1}
N_0={\iota_0}\left(\ds\frac{\text{d}}{\text{d}t}\rho_{0}\right)=\left(
\begin{array}{cc}
-\sigma^x & -\ds\frac{I^x_{0,1;t}(t)}{\tau}\\[9pt]
-\sigma^y & \sigma^x\\
\end{array}\right)
 \in \mathfrak{sl}(2).
\end{equation}
From \eqref{DiffKeqnN}
we may calculate the differential-difference syzygy.
 Equating components in \eqref{DiffKeqnN} and simplifying we obtain
 \begin{equation}\label{taudelta}
 \begin{array}{lcl}
 \ds\frac{\text{d}}{\text{d}t}\kappa&=&\kappa(\mbox{id} -\text{S})\sigma^x+\left(\scalebox{0.9}{\mbox{$\ds\frac1{\tau} -\ds\frac{\tau}{\tau_1^2}$}}\text{S}_2\right)\sigma^y,\\[11pt]
 \ds\frac{\text{d}}{\text{d}t}\tau&=&\tau(\text{S}+\text{id})\sigma^x+\kappa\text{S}\sigma^y
 \end{array}
 \end{equation}
 so that
 \[
\ds\frac{\text{d}}{\text{d}t}\left(
\begin{array}{c}
\kappa\\
\tau\\
\end{array}\right)
=
\mathcal{H}\left(
\begin{array}{c}
\sigma^x\\ \sigma^y
\end{array}\right)
\]
where
 \begin{equation}\label{taudelta2}
 \mathcal{H} =
 \left(\begin{array}{cc}
 \kappa (\mbox{id}-\text{S}) & \scalebox{0.9}{\mbox{$\ds\frac1{\tau} -\ds\frac{\tau}{\tau_1^2}$}} \text{S}_2
 \\[11pt]
 \tau(\mbox{id}+\text{S}) &\kappa\text{S}
 \end{array}\right).
 \end{equation}

 \subsection{The Euler--Lagrange equations and conservation laws}

We are now in a position to obtain the Euler--Lagrange equations and conservation laws for a Lagrangian of the form

\[
\mathcal{L}[x,y]=\sum L(\tau,\tau_1,\dots \tau_{J_1}, \kappa,\kappa_1,\dots, \kappa_{J_2}).
\]

Using Theorem \ref{mainThmELeqns}, we have that the Euler--Lagrange system is $0=\mathcal{H}^* \left( \text{E}_{\kappa}(L)\ \text{E}_{\tau}(L)\right)^T$
which is written explicitly as
\begin{equation}\label{ELex1}
\begin{array}{rcl}
0&=& \left(\mbox{id}-\text{S}_{-1}\right) \kappa\text{E}_{\kappa}(L) + \left( \mbox{id}+\text{S}_{-1}\right)  \tau  \text{E}_{\tau}(L),\\[12pt]
0&=& -\text{S}_{-2}\left(\displaystyle\ds\frac{\tau}{\tau_1^2} \text{E}_{\kappa}(L)\right) +\ds\frac1{\tau}\text{E}_{\kappa}(L)+\text{S}_{-1} \left(\kappa\text{E}_{\tau}(L)\right).
\end{array}
\end{equation}

{ We recall that if $\mathcal{H}=\sum_{k=0}^m c_k \text{S}_k$ then $\mathcal{H}^*=\sum_{k=0}^m (\text{S}_{-k}c_k)\text{S}_{-k}$. Further, we recall the formula
 \[ F\mathcal{H}(G)-\mathcal{H}^*(F)G = (\text{S}-\mbox{id}) A_{\mathcal{H}}\left(F,G\right)\]
 where
 \[A_{\mathcal{H}}\left(F,G\right) = \sum_{k=1}^m \left(\sum_{j=0}^{k-1} \text{S}_j \right)\left( \text{S}_{-k}\left(c_kF\right) G\right)\]
 and where the identity
 \[
(\text{S}_k-\mbox{id})=(\text{S}-\mbox{id}) \sum_{j=0}^{k-1} \text{S}_j
 \]
 has been used.
}

To obtain the conservation laws we need only the boundary terms arising from $\text{E}(L)\mathcal{H}\left( \sigma^x\ \sigma^y\right)^T-\mathcal{H}^*(\text{E}(L))\left( \sigma^x\ \sigma^y\right)^T $, which we record here. They are
$(\text{S}-\mbox{id})A_{\mathcal{H}}$ where 
\begin{equation}\label{eg1Ah} \begin{array}{rcl}A_{\mathcal{H}} &=& \mathcal{C}^x_0\sigma^x + \mathcal{C}^y_0\sigma^y + \mathcal{C}^y_1\text{S}\sigma^y\\[10pt]
&=& \left[-\text{S}_{-1}\left(\kappa \text{E}_{\kappa}(L)\right) + \text{S}_{-1}\left(\tau \text{E}_{\tau}(L)\right)\right]\sigma^x\\[8pt]
&& \ +\left[ \text{S}_{-1}\left(\kappa\text{E}_{\tau}(L)\right) -\text{S}_{-2}\left(\scalebox{.95}{\mbox{$\ds\frac{\tau}{\tau_1^2}$}} \text{E}_{\kappa}(L)\right)\right]\sigma^y \\[10pt]
&&\ -\text{S}_{-1}\left(\scalebox{.95}{\mbox{$\ds\frac{\tau}{\tau_1^2}$}} \text{E}_{\kappa}(L)\right) \text{S}\sigma^y,
\end{array}\end{equation}
where this defines $\mathcal{C}^x_0$, $\mathcal{C}^y_0$ and $\mathcal{C}^y_1$.

To find the conservation laws from $A_{\mathcal{H}}$, we first calculate the invariantized form of the matrix of infinitesimals restricted to the variables $x_0$ and $y_0$ 
\[
\Phi_0(I)=
\begin{blockarray}{cccc}
 & a & b & c \\
\begin{block}{c(ccc)}
  x_0 & 1 & 0 & 0\\
  y_0 & 0 & 0 & 1\\
  \end{block}
\end{blockarray}
\]
and then the replacement required by Theorem \ref{thmi} is given by
\[
\text{S}_k\sigma^x \mapsto \ \left( \begin {array}{ccc} 1&0&0\end{array}\right) \text{S}_k \myAd(\rho_0)
\]
and
\[
\text{S}_k\sigma^y \mapsto \left( \begin {array}{ccc} 0&0&1\end{array}\right) \text{S}_k \myAd(\rho_0).
\]
Since $\text{S}\myAd(\rho_0)=\myAd\left(K_0\right)\myAd(\rho_0)$, 
after collecting terms and simplifying we obtain the Noether's Conservation Laws in the form

\begin{equation}
\label{NCLs}\begin{array}{rcl} {\bf k} &=& \left[ \mathcal{C}^x_0\left( \begin {array}{ccc} 1 & 0 & 0\end{array}\right) + \mathcal{C}^y_0\left( \begin {array}{ccc} 0 & 0 & 1\end{array}\right)  +\mathcal{C}^y_1 \left( \begin {array}{ccc} 0 & 0 & 1\end{array}\right)\myAd(K_0)  \right]\myAd(\rho_0)\\[12pt] &=&V(I)\myAd(\rho_0)\end{array}
\end{equation}
where
\[
\myAd(\rho_0)=\left(\begin{array}{ccc} \ds\frac{x_0y_1 +x_1y_0}{\tau} & \ds\frac{y_0y_1}{\tau} & -\ds\frac{x_0x_1}{\tau} \\[10pt]
 2\ds\frac{x_1y_1}{\tau^2} &  \ds\frac{y_1^2}{\tau^2}  &  -\ds\frac{x_1^2}{\tau^2} \\[10pt]
- 2x_0y_0 & -y_0^2 & x_0^2 
  \end{array}\right)
\]
and
\[
\myAd(K_0)=\left(\begin{array}{ccc} -1 & \kappa\tau & 0\\[10pt] -2\ds\frac{\kappa}{\tau}&\kappa^2& -\ds\frac{1}{\tau^2}\\[10pt] 0 & -\tau^2 & 0\end{array}\right)
\]
and where $\mathcal{C}^x_0$,  $\mathcal{C}^y_0$ and $\mathcal{C}^y_1$ are defined in Equation (\ref{eg1Ah}), the vector  ${\bf k}=(k_1,k_2,k_3)$ is a vector of constants and where this equation defines $V(I)=(V_0^1,V_0^2,V_0^3)^T$.
Explicitly, the vector of invariants $V(I)$ is of the form
 \[
 V(I)={\rm S}_{-1}\left(\begin{array}{ccc} \tau \text{E}_{\tau}(L)-\kappa \text{E}_{\kappa}(L) &  \text{E}_{\kappa}(L) &  \kappa\text{E}_{\tau}(L) - {\rm S}_{-1}\left( \ds\frac{\tau}{\tau_1^2} \text{E}_{\kappa}(L)\right)  \end{array}\right).
 \]
We note that once the Euler--Lagrange equations have been solved for the sequences $(\kappa_k)$ and
$(\tau_k)$, then $ V(I) $ is known, so that (\ref{NCLs}) can be considered as an algebraic equation for $\rho_0$. This will be the focus of the next section.  


Recall that from $(\text{S}-\mbox{id})(V(I)\myAd(\rho_0))=0$ we obtain the discrete Euler--Lagrange equations in the form ${\rm S}V(I)\myAd(\rho_1\rho_0^{-1})=V(I)$ which yields the equations
\begin{equation}\label{vAdeqs}
\left(\begin{array}{ccc} V^1_0 & V^2_0 & V^3_0\end{array}\right) = 
\left(\begin{array}{ccc} V^1_1 & V^2_1 & V^3_1\end{array}\right)\left(\begin{array}{ccc} -1 & \kappa\tau & 0\\[9pt] -2\ds\frac{\kappa}{\tau}&\kappa^2& -\ds\frac{1}{\tau^2}\\[12pt] 0 & -\tau^2 & 0\end{array}\right).
\end{equation}

\subsection{The general solution}
 
 If we can solve for the discrete frame $(\rho_k)$ we then have  from the general theory that
 \[
\left(\begin{array}{c} x_k \\ y_k \end{array}\right)=\rho_k^{-1} \left(\begin{array}{c} 1 \\ 0 \end{array}\right)= \left(\begin{array}{cc} d_k & -b_k \\ -c_k & a_k \end{array} \right)\left(\begin{array}{c} 1 \\ 0 \end{array}\right)=\left(\begin{array}{c} d_k \\ -c_k \end{array}\right)
\]
since the normalisation equations for $\rho_k$ are $\rho_k\cdot ({x}_k, {y}_k)^T=(1,0)^T$.

\begin{theorem}
Given a solution $(\kappa_k)$, $(\tau_k)$ to the Euler--Lagrange equations, so that the vector of invariants ${\rm S}_kV(I)=(V^1_k, V^2_k, V^3_k)^T$ appearing in the conservation laws
are known and satisfy $V^2_k\ne0$ for all $k$, (\ref{NCLs}), and that with the three constants ${\bf k}=(k_1, k_2, k_3)^T$ satisfying $k_3(k_1^2+4k_2k_3)\ne 0$ are given, 
then the general solution to the Euler--Lagrange equations, in terms of $(x_k, y_k)$ is
\[
\left(\begin{array}{c} x_k \\ y_k \end{array}\right)=\left(\begin{array}{cc} 0 & 1 \\ -1 & 0 \end{array} \right)Q\left(\begin{array}{cc }\prod_{l=0}^k \zeta_l \lambda_{1,l} & 0 \\ 0 & \prod_{l=0}^jk \zeta_l \lambda_{2,l}\end{array} \right)
Q^{-1}\left(\begin{array}{c} c_0 \\ d_0 \end{array}\right)
\]
where here, $c_0$ and $d_0$ are two further arbitrary constants of integration,
\begin{equation}\label{Q}
Q=\left( \begin{array}{cc} k_1 - \sqrt{k_1^2+4k_2k_3} &  k_1 + \sqrt{k_1^2+4k_2k_3} \\ 2k_3 & 2k_3 \end{array} \right),
\end{equation}
and where
\begin{equation}\label{lambdaZetaEq}
\lambda_{1,l}=V_l^1-\sqrt{k_1^2+4k_2k_3} ,\qquad \lambda_{2,l}=V_l^1+\sqrt{k_1^2+4k_2k_3},\qquad \zeta_l = -\ds\frac{\tau_l}{2 V^2_l}.
\end{equation}
\end{theorem}

\begin{proof}
If we set
\begin{equation}\label{mfparameters}
\rho_0=\left(\begin{array}{cc} a_0 & b_0 \\ c_0 & d_0 \end{array} \right),\qquad a_0d_0-b_0c_0=1
\end{equation}
and write $\eqref{NCLs}$ in the form ${\bf k}{\myAd(\rho_0)}^{-1}=V(I)$ as three equations for $\{a_0,b_0,c_0,d_0\}$, we obtain
\begin{equation}\label{equatparameters}
\begin{aligned}
(a_0d_0+b_0c_0)k_1+2b_0d_0k_2-2a_0c_0k_3&=V_0^1,\\
c_0d_0k_1+d_0^2k_2-c_0^2k_2&=V_0^2,\\
-a_0b_0k_1-b_0^2k_2+a_0^2k_3&=V_0^3.
\end{aligned}
\end{equation}
Computing a Gr\"{o}ebner basis associated to these equations, together with the equation $a_0d_0-b_0c_0=1$, using   the lexicographic
ordering $k_3 < k_2 < k_1 < c_0 < b_0 < a_0$, 
we obtain 
\begin{subequations}\label{GBSL2}
\begin{align}
 k_1^2+4k_2k_3-{(V_0^1)}^2-4V_0^2V_0^3&=0,\label{GBSL2:1}\\
 k_3c_0^2-k_1c_0d_0-k_2d_0^2+V_0^2&=0,\label{GBSL2:2}\\
 2b_0V_0^2-2c_0k_3+(k_1-V_0^1)d_0&=0,\label{GBSL2:3}\\
 2a_0V_0^2-c_0(k_1+V_0^1)-2k_2d_0&=0.\label{GBSL2:4}
\end{align}
\end{subequations}
We note that (\ref{GBSL2:1}) is a first integral of the Euler--Lagrange equations, (\ref{GBSL2:2}) is a conic equation for $(c_0,d_0)$ while (\ref{GBSL2:3}) and (\ref{GBSL2:4}) are linear for $(a_0,b_0)$ in terms of $(c_0,d_0)$.

We have
\[
\rho_1=\left(\begin{array}{cc} \kappa & \ds\frac{1}{\tau} \\[9pt] -\tau & 0 \end{array}\right) \rho_0
\]
where $\rho_1=\text{S}\rho_0$.
Hence
\[
c_{1}=-\tau a_{0}, \quad \text{and} \quad d_{1}=-\tau b_{0}.
\]
Back-substituting for $a_0$ and $b_0$ from (\ref{GBSL2:3}) and (\ref{GBSL2:4}) yields, assuming $V_0^2\neq0$
\begin{equation}\label{equat5}
\left(\begin{array}{c} c_1 \\ d_1 \end{array}\right)=\ds\frac{-\tau}{2V_0^2}\left(V_0^1{\left( \begin{array}{cc} 1 & 0 \\ 0 & 1\end{array}\right) }+ \left( \begin{array}{cc} {k_1}& 2k_2 \\ 2k_3 & - {k_1}\end{array}\right)\right)\left(\begin{array}{c} c_0 \\ d_0 \end{array}\right).
\end{equation}

Now, setting
\[
\underline{c_0}=\left(\begin{array}{c} c_0 \\ d_0 \end{array}\right), \ \underline{c_1}=\left(\begin{array}{c} c_1 \\ d_1 \end{array}\right), \  \zeta_0=\ds\frac{-\tau}{2V_0^2} \quad \text{and} \quad X_0=V_0^1{\left( \begin{array}{cc} 1 & 0 \\ 0 & 1\end{array}\right)} + 
\left( \begin{array}{cc} {k_1} & 2k_2 \\ 2k_3 & - {k_1}\end{array}\right)
\]
equation \eqref{equat5} can be written as
\begin{equation}\label{equat7}
\underline{c_1}=\zeta_0X_0\underline{c_0}.
\end{equation}
Diagonalising $X_0$ we obtain $\Lambda_0$ diagonal such that
\[
\Lambda_0=Q^{-1}X_0Q=\left(\begin{array}{cc}\lambda^1_0 & 0 \\ 0 & \lambda^2_0  \end{array}\right)
\]
where
\[
\lambda^1_0=V_0^1-\sqrt{k_1^2+4k_2k_3} \quad \text{and} \quad \lambda^2_0=V_0^1+\sqrt{k_1^2+4k_2k_3}
\]
and
\begin{equation}\label{Q1}
Q=\left( \begin{array}{cc} k_1 - \sqrt{k_1^2+4k_2k_3} &  k_1 + \sqrt{k_1^2+4k_2k_3} \\ 2k_3 & 2k_3 \end{array} \right).
\end{equation}
Since $Q$ is a constant matrix, it is now simple to solve the recurrence relation.
From \eqref{equat7}, supposing $k_3\sqrt{k_1^2+4k_2k_3}\neq0$ so $Q^{-1}$ exists, we obtain
\[
\underline{c_{k+1}}=Q\left(\begin{array}{cc }\prod_{l=0}^k \zeta_l \lambda_{1,l} & 0 \\ 0 & \prod_{l=0}^k \zeta_l \lambda_{2,l}\end{array} \right)Q^{-1}\underline{c_{0}}
\]
where here, $\underline{c_{0}}$ is the initial data. Finally from the normalisation equations
\[
\begin{array}{rcl}
\left(\begin{array}{c} x_k \\ y_k \end{array}\right)&=&\rho_k^{-1} \left(\begin{array}{c} 1 \\ 0 \end{array}\right)\\[12pt]&=& \left(\begin{array}{cc} d_k & -b_k \\ -c_k & a_k \end{array} \right)\left(\begin{array}{c} 1 \\ 0 \end{array}\right)\\[12pt]&=&\left(\begin{array}{c} d_k \\ -c_k \end{array}\right) 
\\[12pt]&=&\left(\begin{array}{cc} 0 & 1 \\ -1 & 0 \end{array} \right)\left(\begin{array}{c} c_k \\ d_k \end{array}\right)
\end{array}
\]
and the result follows as required. 
\end{proof}

\begin{remark} The proof of the Theorem makes no use of Equation (\ref{GBSL2:2}). Here we note that it is consistent with  the second component of (\ref{vAdeqs}). We have that (\ref{GBSL2:2}) can be written as
\[
\left(\begin{array}{cc} c_0 & d_0 \end{array} \right) \left(\begin{array}{cc} k_3 & -\ds\frac{k_1}{2}\\  -\ds\frac{k_1}{2} & k_2 \end{array} \right) \left(\begin{array}{cc} c_0 \\ d_0 \end{array} \right) = -V_0^2
\]
and therefore
\begin{equation}\label{equat6}
\left(\begin{array}{cc} c_1 & d_1 \end{array} \right) \left(\begin{array}{cc} k_3 & -\ds\frac{k_1}{2}\\  -\ds\frac{k_1}{2} & k_2 \end{array} \right) \left(\begin{array}{cc} c_1 \\ d_1 \end{array} \right) = -V^2_{1}.
\end{equation}
Substituting \eqref{equat6} into \eqref{equat5} yields after simplification the equation 
$$V_1^2=-\tau^2V_0^3$$ as claimed.
\end{remark}

\section{The $SA(2)=SL(2) \ltimes {\mathbb{R}}^{2}$ linear action}
We write the general element of the equi-affine  group, $SA(2)=SL(2) \ltimes {\mathbb{R}}^{2}$, as $(g,\alpha,\beta)$ where $g\in SL(2)$ as in Equation (\ref{typSL2}), and $\alpha$, $\beta\in\mathbb{R}$.
We then consider the equi-affine group action on $P_n^{(0,0)}(\mathbb{R}^2)$ with coordinates $(x_0,y_0)$ given by
$$ (g,\alpha,\beta)\cdot (x_0,y_0)= (\widetilde{x}_0,\widetilde{y}_0)= (ax_0 +by_0+\alpha, cx_0 +dy_0+\beta), \qquad ad-bc=1.$$
The standard representation of this group is given by
$$(g,\alpha,\beta)\mapsto \left(\begin{array}{cccc} a& b & \alpha\\ c&d&\beta\\0&0&1\end{array}\right).$$
\subsection{The Adjoint action}
The infinitesimals vector fields are of the form
\[
{\bf v}_a=x \partial_{x} - y \partial_{y}, \quad {\bf v}_b=y \partial_{x}, \quad {\bf v}_c=x \partial_{y}, \quad {\bf v}_{\alpha}= \partial_{x}, \quad {\bf v}_{\beta}= \partial_{y}.
\]
We have that the induced action on these vector fields is
\[
\left(\begin{array}{ccccc} \widetilde{{\bf v}_a} & \widetilde{{\bf v}_b} & \widetilde{{\bf v}_c} & \widetilde{{\bf v}_{\alpha}} & \widetilde{{\bf v}_{\beta}} \end{array}\right) =  \left(\begin{array}{ccccc} {\bf v}_a & {\bf v}_b & {\bf v}_c & {\bf v}_{\alpha} & {\bf v}_{\beta} \end{array}\right) {\myAd(g,\alpha, \beta)}^{-1}
\]
where
{\begin{equation}\label{SAadj}
{\myAd(g,\alpha, \beta)}=\begin{blockarray}{cccccc} &a & b & c& {\alpha}& {\beta} \\ \begin{block}{c(ccccc)} a & ad+bc&-ac&bd&0&0\\
 b & -2ab&{a}^{2}&-{b}^{2}&0&0\\
  c & 2cd&-{c}^{2}&{d}^{2}&0&0\\
   {\alpha} & -\alpha(ad+bc) + 2ab \beta &a(c\alpha-a\beta) &b(b\beta-d\alpha)&a&b\\ {\beta} & \beta(ad+bc) - 2cd \alpha &c(c\alpha-a\beta) &d(b\beta-d\alpha)&c&d
\\\end{block}\end{blockarray} \ .
\end{equation}}


{\begin{remark}
We note that \eqref{SAadj} can be written as
\begin{equation}\label{adgk2}
{\myAd(g,\alpha,\beta)}= \left(\begin {array}{cc} \mbox{Id}_3 & 0 \\[15pt] \alpha\left(\begin{array}{ccc} -1 & 0 & 0 \\ 0 & 0 & -1 \end{array}\right)+\beta\left(\begin{array}{ccc} 0 &- 1 & 0 \\ 1 & 0 & 0 \end{array}\right) & \mbox{Id}_2
\end {array}
\right) \left(\begin {array}{cc} {\myAd(g)} &0\\ 0 & g
\end {array}\right)
\end{equation}
where $\mbox{Id}_2$ and $\mbox{Id}_3$ are the $2\times 2$ and $3\times 3$ identity matrices respectively.
\end{remark}
}

\subsection{The discrete frame, the generating invariants and their syzygies}
We define a moving frame $\rho_0$ given by the normalisation equations 
\[
(g,\alpha,\beta)\cdot (x_0, y_0)=(0,0),\quad (g,\alpha,\beta)\cdot (x_{1},y_{1})=(1,0), \quad (g,\alpha,\beta)\cdot (x_{2},y_{2})=(0,*)
\]
where $*$ is to be left free. Solving for the group parameters $a$, $b$, $c$, $d$, $\alpha$ and $\beta$ we obtain
the following standard matrix representation of the moving frame
$$\rho_0= \left( \begin {array}{ccc} {\ds\frac {y_{{2}}-y_{{0}}}{\kappa}}&{\ds\frac {x_{{0}}-x_{
{2}}}{\kappa
}}&{\ds\frac {x_{{2}}y_{{0}}-x_{{0}}y_{{2}}}{\kappa}}\\ \noalign{\medskip}y_{{0}}-y
_{{1}}&x_{{1}}-x_{{0}}&x_{{0}}y_{{1}}-x_{{1}}y_{{0}}
\\ \noalign{\medskip}0&0&1\end {array} \right) 
,$$
where
\[
\kappa = \left( y_{{1}}-y_{{2}} \right) x_{{0}}+\left( y_{{2}} -y_{{0}} \right) x_{{1}}+ \left( y_{{0}}-y_{{1}} \right) x_{{2}}
\]
is an invariant. Indeed,
$$\kappa=\rho_0\cdot y_{2}.$$

We define the discrete moving frame to be $(\rho_k)$ where $$\rho_k=\text{S}_k\rho_0.$$
The Maurer--Cartan matrix in the standard representation is
\begin{equation}\label{CMex1b}
 K_0={\iota_0}(\rho_1)=\left( \begin {array}{ccc} \tau&\ds\frac{1+\tau}{\kappa}& -\tau \\ \noalign{\medskip}-\kappa&-1&
\kappa \\ \noalign{\medskip}0&0&1\end {array} \right)
\end{equation}
where $\kappa=\rho_0\cdot y_2$ is given above, and 
\[
\tau= \ds\frac{x_0(y_1-y_3) +x_1(y_3-y_0)+x_3(y_0-y_1)}{x_1(y_2-y_3)+x_2(y_3-y_1)+x_3(y_1-y_2)}= \ds\frac{\rho_0\cdot y_{3}}{\kappa_1}  \]
where we have used the Replacement Rule, Theorem \ref{reprule}, and where $\kappa_k=\text{S}_k\kappa$.
By the general theory of discrete moving frames the algebra of invariants is generated by $\tau$, $\kappa$ and their shifts. We note that one could take $\rho_0\cdot y_{3}$ and $\kappa$ to be the generators,
but the resulting formulae in the sequel are no simpler.

We obtain
\begin{equation}\label{CMex2b}
N_0={\iota_0}\left(\ds\frac{\text{d}}{\text{d}t}\rho_{0}\right)=  \left( \begin {array}{ccc} \sigma^x-I^x_{0,1;t}(t)&{\ds\frac {\sigma^x-I^x_{0,2;t}(t)}{\kappa}}&-\sigma^x
\\ \noalign{\medskip}\sigma^y-I^y_{0,1;t}(t)
&I^x_{0,1;t}(t)-\sigma^x&-\sigma^y\\ \noalign{\medskip}0&0&0\end {array} \right)
 \end{equation}
 where we have set $\sigma^x \coloneqq I^x_{0,0;t}(t)$ and $\sigma^y \coloneqq I^y_{0,0;t}(t)$.

To obtain $\rho_0\cdot x^{\prime}_j=I^x_{0,j;t}(t)$, $\rho_0\cdot y^{\prime}_j=I^y_{0,j;t}(t)$, $j=1,2$ in terms of $\sigma^x$, $\sigma^y$, $\tau$, $\kappa$ and their shifts, we have, since the translation part of the group plays no role in the action
on the derivatives,
\[\left(\begin{array}{c} \rho_0\cdot x_1^\prime\\\rho_0\cdot y_1^\prime\\ 0\end{array}\right)=\rho_0 \left(\begin{array}{c} x_1^\prime\\y_1^\prime\\ 0\end{array}\right) = \rho_0  \rho_1^{-1} \rho_1 \left(\begin{array}{c} x_1^\prime\\y_1^\prime\\ 0\end{array}\right)=K_0^{-1} \left(\begin{array}{c} \text{S}\sigma^x\\\text{S}\sigma^y\\ 0\end{array}\right)\]
and similarly
\[\rho_0 \left(\begin{array}{c} x_2^\prime\\y_2^\prime\\ 0\end{array}\right) =K_0^{-1} (\text{S}K_0^{-1}) \left(\begin{array}{c} \text{S}_2\sigma^x\\\text{S}_2\sigma^y\\ 0\end{array}\right).\]

Finally from \eqref{DiffKeqnN}
and the relations above, we have the differential-difference syzygy
 \begin{equation}\label{RR3} 
\displaystyle\ds\frac{{\rm d}}{{\rm d}t} \left(\begin{array}{c} \tau\\ \kappa \end{array}\right) = \mathcal{H}\left(\begin{array}{c} \sigma^x \\ \sigma^y\end{array}\right),
\end{equation}
where $\mathcal{H}$ is a difference operator depending only on the generating difference invariants $\tau$, $\kappa$ and their shifts,  and which has the form
\[
\mathcal{H}=\left(\begin{array}{cc} \mathcal{H}_{11} & \mathcal{H}_{12} \\ \mathcal{H}_{21} & \mathcal{H}_{22}\end{array}\right)
\]
where 
\begin{equation}\label{HopsSA2}
\begin{array}{l}
\mathcal{H}_{11}=-\tau+\left(1+\ds\frac{\kappa}{\kappa_1}(1+\tau)\right)\text{S} +\tau \text{S}_2 -\ds\frac{\kappa}{\kappa_1^2}\left[\kappa_2(1+\tau_1)-\kappa_1\right]\text{S}_3,\\[11pt]
\mathcal{H}_{12}=   -\ds\frac{1+\tau}{\kappa} + \ds\frac{\tau(1+\tau_1)}{\kappa_1} \text{S}_2 -\ds\frac{\kappa}{\kappa_1^2\kappa_2}\left[\kappa_2\tau_2(1+\tau_1)-\kappa_1(1+\tau_2)\right]\text{S}_3,\\[11pt]
\mathcal{H}_{21}=-\kappa -\kappa\text{S}+(\tau\kappa_1-\kappa)\text{S}_2,\\[11pt]
\mathcal{H}_{22}=-1-(1+\tau)\text{S} +\left(\tau\tau_1-\ds\frac{\kappa(1+\tau_1)}{\kappa_1}\right)\text{S}_2
.\\
\end{array}
\end{equation}
\subsection{The Euler--Lagrange equations and the conservation laws.}

We consider a Lagrangian of the form $ L(\tau,\dots ,\tau_{J_1}, \kappa,\dots, \kappa_{J_2} )$. Then by {Theorem \ref{mainThmELeqns}}
we have that the Euler--Lagrange equations are
\begin{equation}\label{ASL2ELinv}\begin{array}{rcl}
 0&=& \mathcal{H}_{11}^* \text{E}_{\tau}(L) + \mathcal{H}_{21}^* \text{E}_{\kappa}(L),\\
 0&=& \mathcal{H}_{12}^* \text{E}_{\tau}(L) + \mathcal{H}_{22}^* \text{E}_{\kappa}(L)
 \end{array}
 \end{equation}
 where the $\mathcal{H}_{ij}$ are given in Equation (\ref{HopsSA2}).
 
The boundary terms contributing to the conservation laws are
 \begin{equation}\label{AhSA2}\begin{aligned} A_{\mathcal{H}}&=A_{\mathcal{H}_{11}}(\text{E}_{\tau}(L),\sigma^x)+A_{\mathcal{H}_{21}}(\text{E}_{\kappa}(L),\sigma^x)+\\&+A_{\mathcal{H}_{12}}(\text{E}_{\tau}(L),\sigma^y)+A_{\mathcal{H}_{22}}(\text{E}_{\kappa}(L),\sigma^y)\\
 &= \sum_{k=0}^2 \mathcal{C}^x_k \text{S}_k\sigma^x +  \mathcal{C}^y_k \text{S}_k\sigma^y\\
 \end{aligned}\end{equation}
 where this defines the {$\mathcal{C}^x_k$, $\mathcal{C}^y_k$}.
 {Explicitly 
 \begin{equation}\label{coeffSA}
 \begin{aligned}
 \mathcal{C}^x_0=& {\rm S}_{-1}(1+\ds\frac{\kappa}{\kappa_1}(1+\tau)\text{E}_{\tau}(L)) +{\rm S}_{-2}(\tau\text{E}_{\tau}(L))+{\rm S}_{-3}\left(-\ds\frac{\kappa}{\kappa_1^2}(\kappa_2(1+\tau_1-\kappa_1)\right)\text{E}_{\tau}(L))\\
& \quad + {\rm S}_{-1}(-\kappa\text{E}_{\kappa}(L))+{\rm S}_{-2}((\tau\kappa_1-\kappa)\text{E}_{\kappa}(L)),\\
\mathcal{C}^x_1=& {\rm S}_{-1}(\tau\text{E}_{\tau}(L))+{\rm S}_{-2}\left(-\ds\frac{\kappa}{\kappa_1^2}(\kappa_2(1+\tau_1-\kappa_1)\right)\text{E}_{\tau}(L))+{\rm S}_{-1}((\tau\kappa_1-\kappa)\text{E}_{\kappa}(L)),\\
\mathcal{C}^x_2=&  {\rm S}_{-1}\left(-\ds\frac{\kappa}{\kappa_1^2}(\kappa_2(1+\tau_1-\kappa_1)\right)\text{E}_{\tau}(L)), \\
\mathcal{C}^y_0=&  {\rm S}_{-2}(\ds\frac{\tau(1+\tau_1)}{\kappa_1}\text{E}_{\tau}(L)) + {\rm S}_{-3}(-\ds\frac{\kappa}{\kappa_1^2\kappa_2}(\kappa_2\tau_2(1-\tau_1)-\kappa_1(1+\tau_2))\text{E}_{\tau}(L)) \\
& \quad -{\rm S}_{-1}(1+\tau)\text{E}_{\kappa}(L)) + {\rm S}_{-2}(\tau\tau_1 -\ds\frac{\kappa(1-\tau_1)}{\kappa_1}\text{E}_{\kappa}(L)),\\
\mathcal{C}^y_1=&  {\rm S}_{-1}(\ds\frac{\tau(1+\tau_1)}{\kappa_1}\text{E}_{\tau}(L)) + {\rm S}_{-2}(-\ds\frac{\kappa}{\kappa_1^2\kappa_2}(\kappa_2\tau_2(1-\tau_1)-\kappa_1(1+\tau_2))\text{E}_{\tau}(L)) \\
& \quad + {\rm S}_{-1}(\tau\tau_1 -\ds\frac{\kappa(1-\tau_1)}{\kappa_1}\text{E}_{\kappa}(L)),\\
\mathcal{C}^y_2=& {\rm S}_{-1}(-\ds\frac{\kappa}{\kappa_1^2\kappa_2}(\kappa_2\tau_2(1-\tau_1)-\kappa_1(1+\tau_2))\text{E}_{\tau}(L)).\\
 \end{aligned}
 \end{equation}}

To obtain the conservation laws we need the invariantized form of the matrix of infinitesimals restricted to the variables $x_0$ and $y_0$ 
\[
\Phi_0(I)=\begin{blockarray}{cccccc}
 & a & b & c & \alpha & \beta \\
\begin{block}{c(ccccc)}
  x_0 & 0&0&0&1&0\\
  y_0 & 0&0&0&0&1\\
  \end{block}
\end{blockarray}
\]
and then the replacements required to obtain the conservation laws from $A_{\mathcal{H}}$ are
\[
\text{S}_k\sigma^x \mapsto \left( \begin {array}{ccccc} 0&0&0&1& 0\end{array}\right)\text{S}_k \myAd(\rho_0),\qquad \text{S}_k\sigma^y \mapsto \left( \begin {array}{ccccc} 0&0&0&0&1\end{array}\right)\text{S}_k \myAd(\rho_0).
\]

Hence, the conservation laws are given by $(\text{S}-\text{id})A=0$ where
\begin{equation}\label{ASL2CLs}
\begin{aligned}
A&=\Big[\left( \begin {array}{ccccc} 0&0&0&1&0\end{array}\right)( \mathcal{C}^{x}_0 +\mathcal{C}^{x}_1\myAd(K_0)+\mathcal{C}^{x}_2\myAd(K_0(\text{S}K_0))) \\
&+ \left( \begin {array}{ccccc} 0&0&0&0&1\end{array}\right)( \mathcal{C}^{y}_0 +\mathcal{C}^{y}_1\myAd(K_0)+\mathcal{C}^{y}_2\myAd((\text{S}K_0)K_0)) \Big]\myAd(\rho_0)\\
&=V(I)\myAd(\rho_0)
\end{aligned}
\end{equation}
{where
\[
{\myAd(K_0})= \left(\begin {array}{cc} \mbox{Id}_3 & 0 \\[15pt] \left(\begin{array}{ccc} \tau & -\kappa & 0 \\ \kappa & 0 &  -\tau \end{array} \right)& \mbox{Id}_2
\end {array}
\right) \left(\begin {array}{ccccc} -2\tau & \tau \kappa & -\ds\frac{1+\tau}{\kappa} &0 &0 \\[11pt]
\ds\frac{-2\tau(1+\tau)}{\kappa} & \tau^2 & -\ds\frac{{(1+\tau)}^2}{\kappa^2} & 0 & 0\\[11pt]
2\kappa & -\kappa^2 & 1 & 0 & 0 \\[11pt]
0 & 0 & 0 & \tau & \ds\frac{1+\tau}{\kappa} \\[11pt]
0 & 0 & 0 & -\kappa & -1
\end {array}\right)
\]}
and where this defines the vector of invariants,  $V(I)=\left( V^1_0, V^2_0, V^3_0, V^4_0, V^5_0\right)^T$ and where the $\mathcal{C}^x_j$, $\mathcal{C}^y_j$ are defined in Equation (\ref{AhSA2})  and \eqref{coeffSA}. 

We can thus write the conservation laws in the form
\begin{equation}\label{noetherSA2}
 {\bf k} = V(I)\myAd(\rho_0)
\end{equation}
where ${\bf k}=(k_1,k_2,k_3,k_4,k_5)$ is a vector of constants {and where
\[
{\myAd(\rho_0})= \left(\begin {array}{cc} \mbox{Id}_3 & 0 \\[15pt] \left(\begin{array}{ccc} \ds\frac{x_0y_2-x_2y_0}{\kappa} & x_1y_0-x_0y_1 & 0 \\ x_0y_1-x_1y_0 & 0 &  \ds\frac{x_2y_0-x_0y_2}{\kappa} \end{array} \right)& \mbox{Id}_2
\end {array}
\right) \left(\begin {array}{cc} ({\myAd(g)})\big{|}_{\rho_0} &0\\ 0 & g\big{|}_{\rho_0} 
\end {array}\right)
\]
where
\[
{\myAd(g)})\big{|}_{\rho_0}=\left( \begin{array}{ccc} \ds\frac{(2y_0-y_1-y_2)x_0-(x_1+x_2)y_0+y_2x_1+x_2y_1}{\kappa} & \ds\frac{(y_0-y_2)(y_0-y_1)}{\kappa} & \ds\frac{(x_0-x_2)(x_1-x_0)}{\kappa} \\[11pt]
 \ds\frac{2(y_0-y_2)(x_0-x_2)}{\kappa^2} & \ds\frac{{(y_0-y_2)}^2}{\kappa^2} & -\ds\frac{{(x_0-x_2)}^2}{\kappa^2} \\[11pt]
 2(y_0-y_1)(x_1-x_0) & -{(y_0-y_1)}^2 & {(x_0-x_1)}^2
 \end{array}\right),
\]
and
\[
g\big{|}_{\rho_0}= \left( \begin {array}{cc} {\ds\frac {y_{{2}}-y_{{0}}}{\kappa}}&{\ds\frac {x_{{0}}-x_{
{2}}}{\kappa
}}\\ \noalign{\medskip}y_{{0}}-y
_{{1}}&x_{{1}}-x_{{0}}\end {array} \right).
\]}
We will show in the next section that
a first integral of the Euler--Lagrange equations is given by
\begin{equation}\label{ASL2FirstInt}k_1k_4k_5+k_2k_5^2-k_3k_4^2=V_0^1V_0^4V_0^5+V_0^2{(V_0^5)}^2-V_0^3{(V_0^4)}^2.\end{equation}

\subsection{The general solution}
We now show how to obtain the solution to the Euler--Lagrange equations in terms of the original variables, given the vector of invariants and the constants in the conservation laws, (\ref{noetherSA2}).

\begin{theorem} Suppose a solution $(\tau_k)$, $(\kappa_k)$ to the Euler--Lagrange equations (\ref{ASL2ELinv}), is given, so that the vector of invariants $({\rm S}_kV(I))$  appearing in the conservation laws (\ref{ASL2CLs}) is known, 
and that  $V_0^4V_0^5\ne 0$.
Suppose further that a vector of constants ${\bf k}=(k_1,k_2,k_3,k_4,k_5)$ satisfying $k_4k_5\ne 0$ is given. 
 Then the general solution to the Euler--Lagrange equations, in terms of $(x_k,y_k)$ is given by
\begin{equation}\label{gensolSA2} \left(\begin{array}{c} x_k\\y_k\\1\end{array}\right) = \rho_k^{-1}\left(\begin{array}{c} 0\\0\\1\end{array}\right) =\left(\begin{array}{c}  -\alpha_k d_k +\beta_k b_k \\  \alpha_k c_k -\beta_k c_k  \\1\end{array}\right).\end{equation}
where, setting  $\mu\coloneqq k_1k_4k_5+k_2k_5^2-k_3k_4^2$,
\begin{equation}\label{paramEqsSA2}\begin{aligned}
a_0&=-\ds\frac{V^5_0}{V^4_0} c_0 +\ds\frac{k_4}{V^4_0}\\
b_0&= -\ds\frac{V^5_0k_5}{V^4_0k_4}c_0+\ds\frac{k_4k_5-V^4_0V^5_0}{V^4_0k_4}\\
d_0&=\ds\frac{k_5}{k_4}c_0 +\ds\frac{V^4_0}{k_4}\\
\alpha_0&=\ds\frac{\mu V^5_0}{\left(V^4_0k_4\right)^2} c_0^2 +\ds\frac{\left(k_2k_5^2+k_3k_4^2 +\mu  \right)  V^4_0 (V^5_0)^2 -2\mu k_4k_5V^5_0}{\left(V^4_0k_4\right)^2 V^5_0k_5} c_0\\
&\quad +\ds\frac1{\left(V^4_0k_4\right)^2 V^5_0k_5} \left(k_2k_5\left(V^4_0V^5_0\right)^2 -\left(k_2k_5^2+k_3k_4^2 +\mu   \right)V^4_0V^5_0k_4 +k_4^2k_5\left(V_0^3{(V_0^4)}^2+\mu   \right)  \right)\\
\beta_0 &= -\ds\frac{\mu}{k_4^2 V^4_0} c_0^2 - \ds\frac{V^4_0\left(k_2k_5^2+k_3k_4^2+\mu    \right)}{k_4^2k_5V^4_0} c_0 +\ds\frac{k_4^2V^2_0-k_2(V^4_0)^2}{k_4^2 V^4_0}
\end{aligned}
\end{equation}
and where
\begin{equation}\label{ASL2crsoln}
c_k = \prod_{l=0}^{k-1} \left(\ds\frac{\kappa_l V^5_l}{V^4_l} -1\right) c_0 -\sum_{l=0}^{k-1} \prod_{m=l+1  }^{k-1}  \left(\ds\frac{\kappa_l V^5_m}{V^4_m} -1\right)  \ds\frac{k_4\kappa_l}{V^4_l}\end{equation}
where in this last equation, $c_0$ is the initial datum, or constant of integration. 
\end{theorem}
\begin{proof}
If we can solve for the discrete frame $(\rho_k)$
\[\rho_k=\left(\begin{array}{ccc}a_k & b_k & \alpha_k \\ c_k & d_k & \beta_k\\ 0&0&1\end{array}\right),\] then we have by the normalisation equations, that Equation (\ref{gensolSA2}) holds.
We consider (\ref{noetherSA2})
  as five equations for $\{a_0,b_0,c_0,d_0, \alpha_0, \beta_0\}$, which when written out in detail are
\[
\begin{aligned}
0=&(a_0d_0+b_0c_0)k_1+2b_0d_0k_2-2a_0c_0k_3+(b \beta_0 + d_0 \alpha_0) k_4 - (a_0 \beta_0 + c_0 \alpha_0)-V_0^1,\\
0=&-c_0^2k_3+c_0k_1d_0-c_0k_5\beta_0+k_2d_0^2+k_4d_0k_2-V_0^2,\\
0=&a_0^2k_3-a_0b_0k_1+a_0k_5\alpha_0-b_0^2k_2-b_0k_4\alpha_0-V_0^3,\\
0=&-c_0k_5+k_4d_0-V_0^4,\\
0=&a_0k5-b_0k_4-V_0^5.
\end{aligned}
\]
Computing a Gr\"{o}bner basis associated to these equations  with the lexicographic ordering $k_2<k_1<a_0<b_0<d_0<\beta_0<\alpha_0$,  we obtain the first integral noted in Equation (\ref{ASL2FirstInt}),
and the expressions for $a_0$, $b_0$, $d_0$, $ \alpha_0$ and $\beta_0$ in terms of $c_0$ given in Equations (\ref{paramEqsSA2}), provided $V^4_0$, $V^5_0$, $k_4$ and $k_5$ are all non zero. 

We have $\rho_1=K_0\rho_0$ so that we have a recurrence equation for $(c_k)$, specifically, 
\[ c_1=-\kappa a_0 - c_0= \left(\ds\frac{\kappa V^5_0}{V^4_0} -1\right) c_0 -\ds\frac{k_4\kappa}{V^4_0}\]
where we have back substituted for $a_0$ from $\eqref{paramEqsSA2}$. This is linear and can be easily solved to obtain the expression for $c_k$ given in Equation (\ref{ASL2crsoln}).
Substituting this into the shifts of $\eqref{paramEqsSA2}$ yields $(a_k)$, $(b_k)$, $(d_k)$, $(\alpha_k)$ and $(\beta_k)$
and substituting these into (\ref{gensolSA2}) yields the desired  result.\end{proof}

\section{The $SL(2)$ projective action}

In this example, we show some techniques for calculating the recurrence relations when the action is nonlinear.
We detail the calculations for a class of one-dimensional  $SL(2)$  Lagrangians, which are invariant under 
the projective action of  $SL(2)$. This is defined  by
\begin{equation}\label{sl2actn}
\widetilde{x_0}=g\cdot x_0 =  \ds\frac{ax_0+b}{cx_0+d}, \qquad ad-bc=1.
\end{equation}

\subsection{The Adjoint action}
The infinitesimal vector fields are
\[
{\bf v}_a=2x \partial_{x} , \quad {\bf v}_b= \partial_{x}, \quad {\bf v}_c=-x^2 \partial_{x}.
\]
We have that the induced action on these are

\[
\left(\begin{array}{ccc} \widetilde{{\bf v}_a} & \widetilde{{\bf v}_b} & \widetilde{{\bf v}_c} \end{array}\right) = \left(\begin{array}{ccc} {\bf v}_a & {\bf v}_b & {\bf v}_c \end{array}\right) {\myAd(g)}^{-1} 
\]
where
\begin{equation}
\myAd(g)= \begin{blockarray}{cccc} & a & b & c \\ \begin{block}{c(ccc)} a & ad+bc & -ac & bd\\ b & -2ab & a^{2} & -b^{2}\\ c & 2cd&-c^{2}&d^{2}\\\end{block}\end{blockarray} \
\end{equation}
which matches with \eqref{sl2adj} as expected.

\subsection{The discrete frame, the generating invariants and their syzygies}

We choose the normalisation equations 
$$
\widetilde{x_0}=\ds\frac{1}{2}, \quad \widetilde{x_1}=0, \quad \widetilde{x_2}=-\ds\frac{1}{2}
$$
which we can solve, together with $ad-bc-1$ for $a$, $b$,  $c$ and $d$ to find the frame
\begin{equation}\label{projSL2rho0}
\rho_0 = \ds\frac{\sqrt{x_0-x_2}}{\sqrt{(x_0-x_1)(x_1-x_2)}}
\begin{pmatrix}
\ds\frac{1}{2} & -\ds\frac{x_1}{2} \\[11pt] 
\ds\frac{x_2-2x_1+x_0}{x_0-x_2} & \ds\frac{x_0x_1-2x_{0}x_2+x_1x_2}{x_0-x_2} 
\end{pmatrix}
\end{equation}
and we take $$\rho_k=\text{S}_k\rho_0.$$

\subsection{The generating discrete invariants}
The famous, historical invariant for this action, given four points, is the cross ratio,
\begin{equation}\label{CrossRatio}
\kappa=\ds\frac{(x_0-x_1)(x_2-x_3)}{(x_0-x_3)(x_2-x_1)}.
\end{equation}
By the Replacement Rule, Replacement Rule, Theorem \ref{reprule}, we have that $$\kappa(x_0,x_1,x_2,x_3)=\kappa(\rho_0\cdot x_0,\, \rho_0\cdot x_1,\, \rho_0\cdot x_2,\,  \rho_0\cdot x_3)=\kappa\left(\textstyle\ds\frac12,\, 0,\, -\textstyle\ds\frac12,\, I^x_{0,3}\right)$$ or
$$\kappa = \ds\frac{1+2 I^x_{0,3}}{1-2 I^x_{0,3}}.$$

The Maurer--Cartan matrix is then,
\begin{equation}\label{projSL2MC}
K_0 = {\iota_0}(\rho_1) =\sqrt{\ds\frac{\kappa-1}{4\kappa}}\left(\begin{array}{cc} 1 & \scalebox{0.8}{\mbox{$\ds\frac12$}} \\[12pt] -\ds\frac{6\kappa+2}{\kappa-1} & 1\end{array}\right).
\end{equation}
By the general theory of moving frames, the discrete invariants are generated by $\kappa$ and its shifts.

We now show how the recurrence relations may be obtained for this non-linear action.
\subsection{The generating differential invariants}
We now consider $x_j=x_j(t)$ where $t$ is an invariant parameter. The induced action on these is given by
$$g\cdot x^{\prime}_j= \ds\frac{{\rm d}}{{\rm d}t} g\cdot x_j= \ds\frac{x^{\prime}_j}{(c x_j +d)^2}$$ and hence we have
for $$\rho_k=\left(\begin{array}{cc} a_k & b_k\\ c_k& d_k\end{array}\right)$$ that
$$\rho_k\cdot x^{\prime}_j = \ds\frac{x^{\prime}_j}{(c_k x_j +d_k)^2}.$$
We define
$$ \sigma^x_j \coloneqq \rho_0\cdot x_{j,t}=\ds\frac{x^{\prime}_j}{(c_0 x_j +d_0)^2}=\ds\frac{x^{\prime}_j(x_1-x_0)}{(x_0-x_2)(x_0-x_1)}$$
where $c_0$ and $d_0$ are given in Equation (\ref{projSL2rho0}). In terms of the $\sigma^x_j$, it is straightforward to show
$$N_0={\iota_0}\left(\ds\frac{\text{d}}{\text{d}t}\rho_{0}\right) = \left(\begin{array}{cc} \ds\frac12\sigma^x_1-\ds\frac12\sigma^x_0 & -\sigma^x_1\\ 2\sigma^x_0 -4\sigma^x_1+2\sigma^x_1 & -\ds\frac12\sigma^x_1+\ds\frac12\sigma^x_0\end{array}\right).$$

We now obtain the recurrence relations for the $\sigma^x_j$.
First observe that since $\rho_k\cdot x^{\prime}_j$ is an invariant, we have for all $k$ and $j$ that 
$$\rho_k\cdot x^{\prime}_j = \ds\frac{\widetilde{x^{\prime}_j}}{(\widetilde{c_k} \widetilde{x_j} +\widetilde{d_k})^2}=\ds\frac{\widetilde{x^{\prime}_j}}{(\widetilde{c_k} \widetilde{x_j} +\widetilde{d_k})^2}\Bigg\vert_{\rho_0}
=\ds\frac{\rho_0\cdot {x^{\prime}_j}}{(\widetilde{c_k}\vert_{\rho_0} \rho_0\cdot {x_j} +\widetilde{d_k}\vert_{\rho_0})^2}.$$
Now the frames $\rho_k$ are equivariant, and so $\widetilde{\rho_k}=\rho_k g^{-1}$. Hence
$$ \left(\begin{array}{cc} \widetilde{a_k} &\widetilde{b_k} \\ \widetilde{c_k}  & \widetilde{d_k} \end{array}\right)\Big\vert_{\rho_0} = \rho_k\rho_0^{-1}=\rho_k\rho_{m-1}^{-1}\cdots \rho_1\rho_0^{-1}=(\text{S}^{m-1}K_0)\cdots K_0.$$
In particular, we have
\begin{equation}\label{projSL2syz1}
\text{S}\sigma_0^x=\rho_1\cdot x_1^\prime = \ds\frac{\rho_0\cdot x_1^\prime}{((K_0)_{2,1}\rho_0\cdot x_1 + (K_0)_{2,2})^2} = \ds\frac{4\kappa}{\kappa-1} \sigma^x_1\end{equation}
since $\rho_0\cdot x_1=0$ and  $\rho_0\cdot x_1^\prime=\sigma^x_1$.   Next, 
\begin{equation}\label{projSL2syz2}
\text{S}_2\sigma_0^x=\rho_2\cdot x_2^\prime = \ds\frac{\rho_0\cdot x_2^\prime}{(((\text{S}K_0)K_0)_{2,1}\rho_0\cdot x_2 + ((\text{S}K_0)K_0)_{2,2})^2} = \ds\frac{\kappa_1(\kappa-1)}{(\kappa_1-1)\kappa} \sigma^x_2\end{equation}
noting that by the normalisation equations, $\rho_0\cdot x_1=0$ and $\rho_0\cdot x_2=-1/2$. Similarly, one can prove that
\begin{equation}\label{projSL2syz3}
\text{S}\sigma^x_1 = \ds\frac{\kappa-1}{4\kappa} \sigma^x_2.\end{equation}

We are now ready to calculate the differential difference syzygy. Calculating \eqref{DiffKeqnN} and
 equating components and using the syzygies (\ref{projSL2syz1}), (\ref{projSL2syz1}) and (\ref{projSL2syz3}), we obtain
 \begin{equation}\label{projSL2Heqn}
 \begin{array}{rcl} \ds\frac{{\rm d}}{{\rm d}t} \kappa &= & \ds\frac{\kappa(\kappa-1)\kappa_1(\kappa_2-1)}{\kappa_2(\kappa_1-1)}\text{S}_3\sigma^x_0
 + \ds\frac{\kappa(\kappa_1-1)}{\kappa_1} \text{S}_2\sigma^x_0 - (\kappa-1) \text{S}\sigma^x_0 - \kappa(\kappa-1)\sigma^x_0\\[11pt] &=&\mathcal{H}\sigma^x_0\end{array}
 \end{equation}
where this defines the linear difference operator $\mathcal{H}$.

\subsection{The Euler--Lagrange equations and the conservation laws}

Given a Lagrangian of the form
$$ \mathcal{L}[x] = \sum L(\kappa, \kappa_1, \dots, \kappa_{J})$$
we have from Theorem \ref{mainThmELeqns} 
that the Euler--Lagrange equation is 
$$ 0=\mathcal{H}^*(\text{E}_{\kappa}(L))$$
where $\mathcal{H}$ is given in Equation (\ref{projSL2Heqn}).
Set $\mathcal{H}=\alpha \text{S}_3 + \beta \text{S}_2 + \gamma\text{S}+\delta$ where this defines $\alpha$, $\beta$, $\gamma$ and $\delta$, specifically,
\begin{equation}\label{alphaetcdefs}
\begin{array}{rcl}
\alpha&=& \ds\frac{\kappa(\kappa-1)\kappa_1(\kappa_2-1)}{\kappa_2(\kappa_1-1)},\\[11pt]
\beta&=&\ds\frac{\kappa(\kappa_1-1)}{\kappa_1}, \\[11pt]
\gamma&=&  - (\kappa-1), \\
\delta&=&- \kappa(\kappa-1).\end{array}
\end{equation} 
Then the Euler--Lagrange equation is
$$0=\mathcal{H}^*(\text{E}_{\kappa}(L)) = \text{S}_{-3}(\alpha \text{E}_{\kappa}(L)) + \text{S}_{-2}(\beta \text{E}_{\kappa}(L)) + \text{S}_{-1}(\gamma \text{E}_{\kappa}(L)) + \delta \text{E}_{\kappa}(L).$$

To obtain the conservation law, we need the matrix of infinitesimals, which is
$$\Phi_0=\bordermatrix{ & a & b & c\cr
x_0 & 2x_0 & 1 &-x_0^2}$$ and so $$\Phi_0(I)=\bordermatrix{ & a & b & c\cr
x_0 & 1 & 1 &- \ds\frac14}.$$
Recall the relation
$$\text{E}_{\kappa}(L)\mathcal{H}\sigma^x_0 = \mathcal{H}^*(\text{E}_{\kappa}(L)) \sigma^x_0 +(\text{S}-\mbox{id})(A_{\mathcal{H}}(\text{E}_{\kappa}(L),\sigma^x_0))$$
where for any suitable arguments $F$ and $G$, 
\begin{equation}\begin{aligned}A_{\mathcal{H}}(F,G) = & \left( \text{S}_{-3}(\gamma F) + \text{S}_{-2}(\beta F) + \text{S}_{-1}(\alpha F)\right)G\\
& + \left( \text{S}_{-1}(\beta F) + \text{S}_{-2}(\alpha F)\right) \text{S}(G) + \text{S}_{-1}(\alpha F) \text{S}_2(G).
\end{aligned}
\end{equation}
Hence by Theorem \ref{thmi}, the conservation law is
\begin{equation}\begin{aligned}
\mathbf{k} =& \left( \text{S}_{-1}(\gamma \text{E}_{\kappa}(L)) + \text{S}_{-2}(\beta \text{E}_{\kappa}(L)) + \text{S}_{-1}(\alpha \text{E}_{\kappa}(L))\right)\Phi_0(I)\myAd(\rho_0)\\
& + \left( \text{S}_{-3}(\beta \text{E}_{\kappa}(L)) + \text{S}_{-2}(\alpha \text{E}_{\kappa}(L))\right) \Phi_0(I) \left(\text{S}\myAd(\rho_0)\right) \\ &+ \text{S}_{-1}(\alpha \text{E}_{\kappa}(L))\Phi_0(I) \left(\text{S}_2 \myAd(\rho_0)\right).
\end{aligned}
\end{equation}
Using
$$\begin{aligned}\text{S}\myAd(\rho_0)=& \myAd(\rho_1)=\myAd(K_0)\myAd(\rho_0),\\ \text{S}_2\myAd(\rho_0)=&\myAd(\text{S}(K_0))\myAd(K_0)\myAd(\rho_0)
\end{aligned}$$
and collecting terms,
we arrive at the conservation laws in the form
\begin{equation}\label{projSL2CLveqn} \mathbf{k} =  V(I)\myAd(\rho_0)\end{equation}
where this defines the vector $V(I) = (V^1_0, V^2_0,V^3_0)^T$ {and where
\[
\myAd(\rho_0)=\left(\begin{array}{ccc} \ds\frac{x_1^2 -x_0x_2}{(x_0-x_1)(x_1-x_2)} & \ds\frac{2x_1-x_2-x_0}{2(x_0-x_1)(x_1-x_2)} & \ds\frac{x_1(2x_0x_2-x_1(x_0+x_2))}{2(x_0-x_1)(x_1-x_2)}\\[11pt]
\ds\frac{(x_0-x_2)x_1}{2(x_0-x_1)(x_1-x_2)} & \ds\frac{x_0-x_2}{4(x_0-x_1)(x_1-x_2)} &\ds\frac{x_1^2(x_2-x_0)}{4(x_0-x_1)(x_1-x_2)} \\[11pt]\ds\frac{2(x_2-2x_1+x_0)((x_1-2x_2)x_0+x_1x_2)}{(x_0-x_2)(x_0-x_1)(x_1-x_2)} & - \ds\frac{{(x_2-2x_1+x_0)}^2}{(x_0-x_2)(x_0-x_1)(x_1-x_2)} & \ds\frac{{((x_1-2x_2)x_0+x_1x_2)}^2}{(x_0-x_2)(x_0-x_1)(x_1-x_2)} \end{array}\right).
\]
}
{Explicitly, $V(I)$ is given by
\begin{align*}
V(I)=\left(\begin{array}{ccc} 1 & 1 & \ds\frac{1}{4} \end{array} \right)&\{\left(\text{S}_{-1}(\gamma \text{E}_{\kappa}(L)) + \text{S}_{-2}(\beta \text{E}_{\kappa}(L)) + \text{S}_{-1}(\alpha \text{E}_{\kappa}(L))\right) \\& \qquad + (\text{S}_{-3}(\beta \text{E}_{\kappa}(L)) + \text{S}_{-2}(\alpha \text{E}_{\kappa}(L)))\myAd(K_0)\\
&\qquad  \qquad +\text{S}_{-1}(\alpha \text{E}_{\kappa}(L))\myAd({\rm S}(K_0)) \myAd(K_0) \}
\end{align*}
where
\[
\myAd(K_0)=\left( \begin{array}{ccc} -\ds\frac{\kappa+1}{2\kappa} & \ds\frac{3\kappa + 1}{2\kappa} & \ds\frac{\kappa-1}{8\kappa} \\[11pt] \ds\frac{-\kappa + 1}{4\kappa} & \ds\frac{\kappa - 1}{4\kappa} & \ds\frac{-\kappa+1}{16\kappa} \\[11pt] -\ds\frac{3\kappa + 1}{\kappa} & -\ds\frac{{(3\kappa + 1)}^2}{(\kappa-1)\kappa} & \ds\frac{\kappa-1}{4\kappa}\end{array}\right)
\]
and where
\[
\myAd({\rm S}(K_0))\myAd(K_0) = \ds\frac{1}{\kappa \kappa_1}\left(\begin{array}{ccc} \ds\frac{(1-\kappa_1)\kappa + \kappa_1+1}{2} & \ds\frac{(1-3\kappa_1)\kappa^2 + 2(1-\kappa_1)\kappa + 1}{2(\kappa-1)} & \ds\frac{(\kappa_1+1)(1-\kappa)}{8}\\[11pt]
\ds\frac{(\kappa_1-1)(\kappa+1)}{4} & \ds\frac{(\kappa_1-1){(\kappa+1)}^2}{4(\kappa -1)} & \ds\frac{(\kappa_1+1)(1-\kappa)}{16}\\[11pt]
\ds\frac{(3\kappa -1)\kappa_1^2 + 2(\kappa-1)\kappa_1 - \kappa - 1}{\kappa_1 -1} & -\ds\frac{{(\kappa_1(3\kappa-1)-\kappa-1)}^2}{(\kappa_1-1)(\kappa-1)} & \ds\frac{(\kappa-1){(\kappa_1 + 1)}^2}{4(\kappa_1-1)}
\end{array}\right).
\]
}

\subsection{The general solution}

If we can solve for the discrete frame $(\rho_k)$ then we have
\begin{equation}\label{projxfinal}x_k = \rho_k^{-1} \cdot \ds\frac12 = \ds\frac{d_k-2 b_k}{2 a_k - c_k}\end{equation}
since $\rho_k\cdot x_k = \textstyle\ds\frac12$ is the normalisation equation.

The Adjoint representation of the Lie group $G$ is, in this case, precisely that of the linear action discussed in \S2, and so we may make use of the simplification of the algebraic equations for the 
group parameters stemming from the conservation laws, $\mathbf{k} = V(I)\myAd(\rho_0)$, that we described there. However, the Maurer--Cartan matrix is different, and so the recurrence relations needed to complete the solution, differ. Nevertheless, we again find that the remaining recurrence relations are diagonalisable, and are therefore easily solved.

\bigskip

We again have equations (\ref{GBSL2:1})--(\ref{GBSL2:4}), where now the $V_0^i$ are those of
equation (\ref{projSL2CLveqn}); we rewrite these here for convenience,
\begin{subequations}\label{projGBSL2}
\begin{align}
 k_1^2+4k_2k_3-{(V_0^1)}^2-4V_0^2V_0^3&=0,\label{projGBSL2:1}\\
 k_3c_0^2-k_1c_0d_0-k_2d_0^2+V_0^2&=0,\label{projGBSL2:2}\\
 2b_0V_0^2-2c_0k_3+(k_1-V_0^1)d_0&=0,\label{projGBSL2:3}\\
 2a_0V_0^2-c_0(k_1+V_0^1)-2k_2d_0&=0.\label{projGBSL2:4}
\end{align}
\end{subequations}
The recurrence relation is $\rho_1=K_0 \rho_0$ or
$$ \left(\begin{array}{cc} a_1 & b_1\\ c_1 & d_1\end{array}\right)= \sqrt{\ds\frac{\kappa-1}{4\kappa}}\left(\begin{array}{cc} 1 & \scalebox{0.8}{\mbox{$\ds\frac12$}} \\[12pt] -\ds\frac{6\kappa+2}{\kappa-1} & 1\end{array}\right) \left(\begin{array}{cc} a_0 & b_0\\ c_0 & d_0\end{array}\right)    $$
leading to the equations
$$c_1=\sqrt{\ds\frac{\kappa-1}{4\kappa}} \left( -\ds\frac{6\kappa+2}{\kappa-1} a_0 +c_0\right),\qquad d_1=-\sqrt{\ds\frac{\kappa-1}{4\kappa}} \left(  -\ds\frac{6\kappa+2}{\kappa-1} b_0 +d_0\right).$$
Using these to eliminate $a_0$ and $b_0$ from (\ref{projGBSL2:3}) and (\ref{projGBSL2:4}), leads to the linear system,
$$\left(\begin{array}{c} c_1 \\ d_1\end{array}\right) = Q \Lambda_0 Q^{-1} \left(\begin{array}{c} c_0 \\ d_0\end{array}\right) $$
where $Q$ is a constant matrix and $\Lambda_0$ is diagonal. Indeed, setting $$\mu=\sqrt{k_1^2+4k_2k_3}$$ we have
$$ Q =\ds\frac1{2\mu} \left(\begin{array}{cc} \mu + k_1  &\mu-k_1\\ 2k_3 & -2 k_3\end{array}\right).$$
Further, $\Lambda_0=\mbox{diag}(\lambda_0^1,\lambda_0^2)$ where
$$\begin{array}{rcl}
\lambda_0^1 &=& \ds\frac1{2\sqrt{\kappa-1}\,\sqrt{\kappa}\,V^2_0} \left(-(3\kappa+1)\left(\mu +V^1_0\right)+(\kappa-1)V^2_0 \right),\\[11pt]
\lambda_0^2 &=& \ds\frac1{2\sqrt{\kappa-1}\,\sqrt{\kappa}\,V^2_0} \left((3\kappa+1)\left(\mu-V^1_0\right)+(\kappa-1)V^2_0 \right).
\end{array}$$
We have then that
$$\left(\begin{array}{c} c_k \\ d_k \end{array}\right) = Q \left(\begin{array}{cc} \prod_{l=0}^{k-1}\lambda_l^1&0\\0& \prod_{l=0}^{k-1}\lambda_l^2\end{array}\right) Q^{-1}\left(\begin{array}{c} c_0 \\ d_0\end{array}\right) $$
where in this last, $c_0$ and $d_0$ are the initial data and $\lambda^l_k=\text{S}_k\lambda^l_0$.

Substituting these into the $k^{\mbox{th}}$ shifts of (\ref{projGBSL2:3}) and (\ref{projGBSL2:4}), specifically,
\begin{subequations}\label{projGBSL2shift}
\begin{align}
 2b_kV_k^2-2c_kk_3+(k_1-V_k^1)d_k&=0,\label{projGBSL2r:3}\\
 2a_kV_k^2-c_k(k_1+V_k^1)-2k_2d_k&=0\label{projGBSL2r:4}
\end{align}
\end{subequations}
 yields the expressions for $a_k$ and $b_k$ needed to obtain, finally, $x_k$ given in (\ref{projxfinal}).

\section{Conclusions}
In these two papers we have introduced difference moving frames and applications to variational problems. We have also shown how to use these frames to solve Lie group invariant difference (recurrence) relations. We have considered relatively simple, solvable Lie group actions and some $SL(2)$ actions. Open problems include the efficient use of difference frames for numerical approximations which preserve symmetry. 


\section*{Acknowledgements}
The authors would like to thank the SMSAS at the University of Kent and the EPSRC (grant EP/M506540/1) for funding this research.

\bibliographystyle{agsm}

\end{document}